\documentclass[11pt,reqno]{amsart}

\usepackage{amsmath,amssymb,color,mathtools}
\usepackage[margin=4cm]{geometry}
\usepackage{cases}
\usepackage{tikz}
\usepackage{hyperref}

\newtheorem{lemma}{Lemma}[section]

\newtheorem{defn}[lemma]{Definition}
\newtheorem{notation}[lemma]{Notation}

\newtheorem{thm}[lemma]{Theorem}
\newtheorem{example}[lemma]{Example}
\newtheorem{thm?}[lemma]{Theorem?}

\begin{document}
\title[Zeros of polynomials over finite Witt rings]{Zeros of polynomials over finite Witt rings}

\author{Weihua Li}
\author{Wei Cao}

\theoremstyle{definition}
\newtheorem{dfn}[lemma]{Definition}
\renewcommand\atop[2]{\genfrac{}{}{0pt}{}{#1}{#2}}
\newcommand{\Mod}[1]{\ (\mathrm{mod}\ #1)}
\newcommand{\etalchar}[1]{$^{#1}$}
\newcommand{\F}{\mathbb{F}}
\newcommand{\Z}{\mathbb{Z}}
\newcommand{\N}{\mathbb{N}}
\newcommand{\C}{\mathbb{C}}
\newcommand{\Q}{\mathbb{Q}}
\newcommand{\R}{\mathbb{R}}
\newcommand{\et}{\textrm{\'et}}
\newcommand{\ra}{\ensuremath{\rightarrow}}
\newcommand{\lra}{\ensuremath{\longrightarrow}}
\newcommand{\FF}{\F}
\newcommand{\ff}{\mathfrak{f}}

\newcommand{\NN}{\widetilde{\N}}
\newcommand{\mm}{\underline{m}}
\newcommand{\nn}{\underline{n}}
\newcommand{\ch}{}

\renewcommand{\P}{\mathbb{P}}
\newcommand{\PP}{\mathbb{P}}
\newcommand{\pp}{\mathfrak{p}}

\newcommand{\ab}{\operatorname{ab}}
\newcommand{\Aut}{\operatorname{Aut}}
\newcommand{\gk}{\mathfrak{g}_K}
\newcommand{\gq}{\mathfrak{g}_{\Q}}
\newcommand{\OQ}{\overline{\Q}}
\newcommand{\Out}{\operatorname{Out}}
\newcommand{\End}{\operatorname{End}}
\newcommand{\Gal}{\operatorname{Gal}}
\newcommand{\CT}{(\mathcal{C},\mathcal{T})}
\newcommand{\lcm}{\operatorname{lcm}}
\newcommand{\Div}{\operatorname{Div}}
\newcommand{\OO}{\mathcal{O}}
\newcommand{\rank}{\operatorname{rank}}
\newcommand{\tors}{\operatorname{tors}}
\newcommand{\IM}{\operatorname{IM}}
\newcommand{\CM}{\mathbf{CM}}
\newcommand{\HS}{\mathbf{HS}}
\newcommand{\Frac}{\operatorname{Frac}}
\newcommand{\Pic}{\operatorname{Pic}}
\newcommand{\coker}{\operatorname{coker}}
\newcommand{\Cl}{\operatorname{Cl}}
\newcommand{\loc}{\operatorname{loc}}
\newcommand{\GL}{\operatorname{GL}}
\newcommand{\PGL}{\operatorname{PGL}}
\newcommand{\PSL}{\operatorname{PSL}}
\newcommand{\Frob}{\operatorname{Frob}}
\newcommand{\Hom}{\operatorname{Hom}}
\newcommand{\Coker}{\operatorname{\coker}}
\newcommand{\Ker}{\ker}
\newcommand{\g}{\mathfrak{g}}
\newcommand{\sep}{\operatorname{sep}}
\newcommand{\new}{\operatorname{new}}
\newcommand{\Ok}{\mathcal{O}_K}
\newcommand{\ord}{\operatorname{ord}}
\newcommand{\Ohell}{\OO_{\ell^{\infty}}}
\newcommand{\cc}{\mathfrak{c}}
\newcommand{\ann}{\operatorname{ann}}
\renewcommand{\tt}{\mathfrak{t}}
\renewcommand{\cc}{\mathfrak{a}}
\renewcommand{\aa}{\mathfrak{a}}
\newcommand\leg{\genfrac(){.4pt}{}}
\renewcommand{\gg}{\mathfrak{g}}
\renewcommand{\O}{\mathcal{O}}
\newcommand{\Spec}{\operatorname{Spec}}
\newcommand{\rr}{\mathfrak{r}}
\newcommand{\rad}{\operatorname{rad}}
\newcommand{\SL}{\operatorname{SL}}
\newcommand{\fdeg}{\operatorname{fdeg}}
\renewcommand{\rank}{\operatorname{rank}}
\newcommand{\Int}{\operatorname{Int}}
\newcommand{\zz}{\mathbf{z}}

\begin{abstract}
Let $\mathbb{F}_q$ denote the finite field of characteristic $p$ and order $q$. Let $\mathbb{Z}_q$ denote the unramified extension of the $p$-adic rational integers $\mathbb{Z}_p$ with residue field $\mathbb{F}_q$. Given two positive integers $m,n$, define a box $\mathcal B_m$ to be a subset of $\Z_q^n$ with $q^{nm}$ elements such that $\mathcal B_m$ modulo $p^m$ is equal to $(\Z_q/p^m \Z_q)^n$. For a collection of nonconstant polynomials $f_1,\dots,f_s\in \mathbb{Z}_q[x_1,\ldots,x_n]$ and positive integers $m_1,\dots,m_s$, define the set of common zeros inside the box $\mathcal B_m$ to be
$$V=\{X\in \mathcal B_m:\; f_i(X)\equiv 0\mod {p^{m_i}}\mbox{ for all } 1\leq i\leq s\}.$$
It is an interesting problem to give the sharp estimates for the $p$-divisibility of $|V|$. This problem has been partially solved for the three cases: (i) $m=m_1=\cdots=m_s=1$, which is just the Ax-Katz theorem, (ii) $m=m_1=\cdots=m_s>1$, which was solved by Katz, Marshal and Ramage, and (iii) $m=1$, and $ m_1,\dots,m_s\geq 1$, which was recently solved by Cao, Wan and Grynkiewicz. Based on the multi-fold addition and multiplication of the finite Witt rings over $\F_q$, we investigate the remaining unconsidered case of $m>1$ and $m\neq m_j$ for some $1\leq j\leq s$, and finally provide a complete answer to this problem. 
\end{abstract}

\maketitle

\tableofcontents

\section{Introduction} \label{sect1}
\subsection{Notation and problem}

Throughout, let $\N$ denote the set of positive integers and $\N_0=\{0\}\cup\N$, let $p$ be a prime number and $q=p^h$ with $h\in \N$. Let $\mathbb{F}_q$ be the finite field of order $q$.
Let $\Q_p$ be the field of $p$-adic rational numbers and $\Z_p$ the ring of integers in $\Q_p$. Let $\Q_q$ be the unramified extension of $\Q_p$ of degree $h$ and $\Z_q$ the ring of integers in $\Q_q$. The residue field of $\Z_q$ is $\F_q$, i.e., $\Z_q/p\Z_q\cong \F_q$. Let $T_q$ be the set of Teichm\"{u}ller representatives of $\mathbb{F}_{q}$ in $\mathbb{Z}_q$. Then each element $a\in \Z_q$ can be uniquely expressed as $a=\sum_{i=0}^{\infty}a_ip^i$ with $a_i\in T_q$.

Given $m,n\in \N$, we define a box $\mathcal B_m$ to be a subset of $\Z_q^n$ with $q^{nm}$ elements such that $\mathcal B_m$ modulo $p^m$ is equal to $(\Z_q/p^m \Z_q)^n$. The box $\mathcal B_m$ is called a split box if $\mathcal B_m = \mathcal I_1 \times \cdots \times \mathcal I_n$, where each $\mathcal I_i$ is an $m$-dimensional box in $\Z_q$ lifting $\mathbb{F}_q$. If each element $a\in \mathcal B_m$ has the expression as $a=\sum_{i=0}^{m-1}a_ip^i$ with $a_i\in T_q^n$, we call $\mathcal B_m$ the Teichm\"{u}ller box, and simply denote by $\mathcal B_m=\mathcal T_m$.

Let $\mathbb{Z}_q[x_1,\ldots,x_n]$ denote the ring of the polynomials in $n$ variables $x_1,\ldots,x_n$ with coefficients in $\mathbb{Z}_q$. Let $f_1,\dots,f_s\in \mathbb{Z}_q[x_1,\ldots,x_n]$ be a system of nonconstant polynomials. Write $X:=(x_1,\ldots,x_n)$ and $[a,b]:=\{ x\in \Z \mid a\leq x\leq b\}$ for $a,b\in \R$. As usual, let $|S|$ denote the cardinality of a set $S$. Given $s$ positive integers $m_1,\dots,m_s$, define
\begin{align*}V=\{X\in \mathcal B_m:\; f_k(X)\equiv 0\mod {p^{m_k}}\mbox{ for all } k\in [1,s]\}.\end{align*}

We are interested in the $p$-adic estimate of $|V|$. Let $\mathrm{ord}_q $ denote the $q$-adic additive
valuation normalized by $\mathrm{ord}_qq=1$. If $q=p$, then $\mathrm{ord}_p $ is the $p$-adic additive
valuation normalized by $\mathrm{ord}_pp=1$. We want to know the $p$-adic estimate of $|V|$, in other words, what is the sharp lower bound for $\mathrm{ord}_q(|V|)$? The satisfactory $p$-adic estimates of $|V|$ have been obtained separately for the following three cases:
\begin{itemize}
  \item [(a)] $m=m_1=\cdots=m_s=1$,
  \item [(b)] $m=m_1=\cdots=m_s>1$,
  \item [(c)] $m=1, m_1,\dots,m_s\geq 1$.
\end{itemize}
Here the adjective ``satisfactory" means that the lower bound can be achieved. It will be seen that the expressions of the $p$-adic estimates of $|V|$ in the three cases above are quite different, especially in the case (b). The reason lies in two aspects: (\romannumeral1) when $m>1$, the residue ring $\Z_q/p^m\Z_q$ will no longer be a field, and (\romannumeral2) the introduction of the box makes the problem much more complicated.

The remaining case of $m>1$ and $m\neq m_j$ for some $j\in[1,s]$ of this problem has not been considered. This paper will be addressed to this problem, including the remaining unconsidered case. We will also give the satisfactory $p$-adic estimate for this case in general and finally provide a complete answer to this problem.

\subsection{Partial known results}
According to the historical development on this topic, the result for a single polynomial, the simplest case, has first been found, which has then been extended to a system of polynomials.

If $m=m_1=\cdots=m_s=1$, then $\Z_q/p\Z_q\cong\F_q$ and hence the problem reduces to the classical one of counting points on an algebraic variety over the finite field $\F_q$. For $t\in\mathbb{R}$, let $\lceil t \rceil^*$ denote the least nonnegative integer more than or equal to $t$, and let $\lfloor t \rfloor$ denote the greatest integer less then or equal to $t$. The following theorem, restated in our setting, was first found by Ax \cite{Ax} for $s=1$ and then by N.M. Katz \cite{Katz} for general $s\geq 1$.
\begin{thm}[Ax-Katz]\label{axkatzthm}
Let $p$ be a prime number and $q = p^h$ with $h\in\mathbb{N}$. Let $\mathcal B_1$ be a subset of $\Z_{q}^n$ with $q^n$ elements such that $\mathcal B_1$ modulo $p$ is equal to $\F_{q}^n$. Let $f_1,\dots,f_s\in \mathbb{Z}_q[x_1,\ldots,x_n]$ be a system of nonconstant polynomials. Let $V=\{X\in \mathcal B_1:f_1(X)\equiv \cdots\equiv f_s(X)\equiv 0\mod {p}\}$. 
Then
\begin{equation*}
  \mathrm{ord}_q(|V|)\geq \left\lceil\frac{n-\sum_{k=1}^s\deg f_k}{\max\nolimits_{k \in [1,s]}\deg f_k}\right\rceil^*.
\end{equation*}
\end{thm}
The Ax-Katz theorem greatly improves the original Chevalley-Warning theorem, which asserts that $|V|$ is divisible by $p$ provided that $n>\sum_{k=1}^s\deg f_k$. There have been a number of work devoting to this case (see, for example, \cite{AS,AM,clark2,cao1,cao3,cao4,caosun,caowan,castro,chencao,Clark1,clark3,Moreno,moreno2,Wilson}).

The case of $m=m_1=\cdots=m_s>1$ was first tackled and solved by Marshal and Ramage \cite{MR} for one polynomial, and then was extended to the system of polynomials by D.J. Katz \cite{DJKatz} in their study on zeros of polynomials over finite principal ideal rings. Note that the residue ring $\Z_q/p^m\Z_q$ is a special kind of finite principal ideal rings, which is isomorphic to the Galois ring (cf. Subsection \ref{boxsub}). Below are their results restated in our setting.
\begin{thm}[Katz-Marshal-Ramage]\label{mrkthm}
Let $p$ be a prime number and $q = p^h$ with $h\in\mathbb{N}$. Let $m\geq 2$ be an integer and $\mathcal B_m$ be a subset of $\Z_{q}^n$ with $q^{mn}$ elements such that $\mathcal B_m$ modulo $p^m$ is equal to $(\Z_{q}/p^m\Z_{q})^n$. Let $f_1,\dots,f_s\in \mathbb{Z}_q[x_1,\ldots,x_n]$ be a system of nonconstant polynomials. Let $V=\{X\in \mathcal B_m: f_1(X)\equiv \cdots\equiv f_k(X)\equiv 0\mod {p^{m}}\}$. 
Then
\begin{equation*}
  \mathrm{ord}_q(|V|)\geq\left\{
                           \begin{array}{ll}
                           \left\lfloor\frac{(n-s+1)m-1}{2}\right\rfloor & \hbox{  if  }n>s \hbox{ and  any }\deg f_k>1,\\
                           \left\lceil (n-s)m\right\rceil^* & \hbox{otherwise.}
                           \end{array}
                         \right.
\end{equation*}
\end{thm}

For the case of $m\neq m_j$ for some $j\in[1,s]$, it turns out that in general the answer depends heavily upon the box $\mathcal B_m$ chosen. Roughly speaking, the effective $p$-estimate of $|V|$ can be obtained only if the box $\mathcal B_m$ has a low algebraic relation with the Teichm\"{u}ller box $\mathcal T_m$ in some sense. Here we simply use the notation $\mathcal B_m \overset{m_s}{\sim} \mathcal T_m$ to denote this relation (cf. Section \ref{sect3} for the detail explanation).

\begin{thm}[Cao-Wan-Grynkiewicz]\label{caowanthm}
Let $p$ be a prime number and $q = p^h$ with $h\in\mathbb{N}$. Let $m_1,\dots,m_s\in \N$ with $m_s\geq \dots \geq m_1$. Let $\mathcal B_1$ be a subset of $\Z_{q}^n$ with $q^n$ elements such that $\mathcal B_1$ modulo $p$ is equal to $\F_{q}^n$. Let $f_1,\dots,f_s\in \mathbb{Z}_q[x_1,\ldots,x_n]$ be a system of nonconstant polynomials. Let $V=\{X\in \mathcal B_1: f_k(X)\equiv 0\mod {p^{m_k}}\mbox{ for all } k\in [1,s]\}$. 
If $\mathcal B_1 \overset{m_s}{\sim} \mathcal T_1$, then
\begin{equation}\label{intreq1}
\mathrm{ord}_q(|V|)\geq \left\lceil\frac{n-\sum_{k=1}^s\frac{p^{m_k}-1}{p-1}\deg f_k}{\max\nolimits_{k \in [1,s]}\{p^{m_k-1}\deg f_k\}}\right\rceil^*.
\end{equation}
In particular, if $q=p$, then \eqref{intreq1} holds for all split boxes $\mathcal B_1$.
\end{thm}
Cao and Wan \cite{caowan} proved Theorem \ref{caowanthm}, which generalizes and improves the result by Grynkiewicz \cite{Grynkiewicz} for the case that $q=p$ and $\mathcal B_1$ is split. In fact, the stronger result was obtained by Cao and Wan \cite[Theorem 4.6]{caowan} in which the Teichm\"uller expansions of polynomials over $\Z_q$ are taken into consideration. 

\subsection{Complete answer to the problem}

The theorem below, which is the same as Theorem \ref{thmforsystem}, provides a complete answer to the problem as mentioned above. 

\begin{thm}
Let $p$ be a prime number and $q = p^h$ with $h\in\mathbb{N}$. Let $m,m_1,\dots,m_s\in \N$ with $m_s\geq \dots \geq m_1$. Let $\mathcal B_m$ be a subset of $\Z_{q}^n$ with $q^{mn}$ elements such that $\mathcal B_m$ modulo $p^m$ is equal to $(\Z_{q}/p^m\Z_{q})^n$. Let $f_1,\dots,f_s\in \mathbb{Z}_q[x_1,\ldots,x_n]$ be a system of nonconstant polynomials. Let $V=\{X\in \mathcal B_m: f_k(X)\equiv 0\mod {p^{m_k}}\mbox{ for all } k\in [1,s]\}$. Assume $\mathcal B_m \overset{m_s}{\sim} \mathcal T_m$. In general, we have
\begin{equation*}
  \mathrm{ord}_q(|V|)\geq\left\lceil\frac{nm-\sum_{k=1}^s\frac{p^{m_k}-1}{p-1}\deg f_k}{\max\nolimits_{k \in [1,s]}\{p^{m_k-1}\deg f_k\}}\right\rceil^*.
\end{equation*}
In particualr, further assume $m\geq m_s=\dots=m_1$, then
\begin{numcases}{\mathrm{ord}_q(|V|)\geq}
 \left\lceil\frac{n-\sum_{k=1}^s\deg f_k}{\max\nolimits_{k \in [1,s]}\deg f_k}\right\rceil^*+n(m-m_1) & if   $m_1=1$,\nonumber \\
  \left\lfloor\frac{(n-s+1)m_1-1}{2}\right\rfloor+n(m-m_1) &  if $m_1>1, n> s, \deg f_1>1,\dots, \deg f_s> 1$,\nonumber \\
     \left\lceil(n-s)m_1\right\rceil^*+n(m-m_1) & \mbox{otherwise.} \nonumber
\end{numcases}
\end{thm}

This paper is organized as follows. In Section \ref{sect2}, we review some necessary facts and properties about the Witt vectors, finite Witt rings, Galois rings and their relationships with the ring $\Z_q$, with focus on the multi-fold addition and multiplication in these rings as well as the estimates for the degrees of associated polynomials for summation and production. The main results are given in Section \ref{sect3}, in which we first investigate the single polynomial case (cf. Theorem \ref{thmforsing}), then we extend it to the polynomial system (cf. Theorem \ref{thmforsystem}). Note that both Theorems \ref{thmforsing} and \ref{thmforsystem} for the case of $m>1$ and $m\neq m_j(j\in[1,s])$ can be improved by using the tighter upper bounds for the degrees of the terms in the Teichm\"uller expansions of polynomials. Some examples and future research are provided in Section \ref{sect4}. 

\section{Preliminaries on Witt rings}\label{sect2}

The tools to study our problem are Witt vectors and Teichmüller lifting, which also play important roles in localization theory and many other modern topics in mathematics. In this section, we only review the construction and properties of the classical $p$-typical Witt vectors. For more other generalizations of Witt vectors, refer to \cite{rabinoff,serre}.

\subsection{Construction of Witt rings}
Let $(X_0,X_1,\dots)$ and $(Y_0,Y_1,\dots)$ be two sequences of indeterminates. For each $k\in\mathbb{N}_0$, the $k$-th Witt polynomial is defined to be
\begin{equation}\label{wittpoly}
w_k(X_0,X_1,\dots) := \sum_{i=0}^kp^iX_i^{p^{k-i}} = X_0^{p^k} + pX_1^{p^{k-1}} + \cdots + p^kX_k.
\end{equation}
\begin{thm}\cite[Theorem 6, pp. 40]{serre}\label{theorem:Witt:op:poly}
    For every $\Phi\in\mathbb{Z}[X,Y]$,
    there exist a unique sequence $(\phi_0, \phi_1,\ldots)$ of elements of $\mathbb{Z}[X_0, Y_0, X_1, Y_1, \dots]$ such that for all $k\in\mathbb{N}_0$, 
    $$w_k(\phi_0,\phi_1,\dots) = \Phi(w_k(X_0,X_1,\dots), w_k(Y_0,Y_1,\dots)).$$
\end{thm}

Let $\Phi_1 = S(X,Y) = X+Y$ and $\Phi_2= M(X,Y) = XY$. By Theorem \ref{theorem:Witt:op:poly} there exists unique sequences of polynomials $(S_0,S_1,\dots)$ and $(M_0,M_1,\dots)$ such that for all $k\in\mathbb{N}_0$,
\begin{align*}
    w_k(S_0,S_1,\dots) &= w_k(X_0,X_1,\dots)+w_k(Y_0,Y_1,\dots), \\
    w_k(M_0,M_1,\dots) &= w_k(X_0,X_1,\dots)w_k(Y_0,Y_1,\dots).
\end{align*}
It is easy to calculate that
$$S_0 = X_0+Y_0, \quad S_1 = -\sum_{i=1}^{p-1}\frac{1}{p}\binom{p}{i}X_0^{i}Y_0^{p-i} + X_1+Y_1,$$
$$M_0 = X_0Y_0, \quad\quad M_1 = X_0^pY_1+X_1Y_0^p+pX_1Y_1.\quad\quad\quad\quad$$
The formulae $S_k$ and $M_k$ for $k>1$ are more complicated. Fortunately, due to our problem,
we do not need to know the explicit expressions of them except for their degrees.

Using the polynomials $w_k$, $S_k$ and $M_k$, we can construct the Witt ring. Let $R$ be a commutative ring with identity. The set of Witt vectors over $R$ is $$W(R)= \{ (a_1,a_2,\cdots)| a_i\in R \text{~for all~} i\in \mathbb{N}_0\}.$$
Define the ghost map as
\begin{align*}
    \omega: W(R) &\longrightarrow R^{\mathbb{N}_0}\\
    \boldsymbol{a} = (a_0,a_1,\dots) &\longmapsto \omega(\boldsymbol{a}) = (w_0(a_0,\dots), w_1(a_0,a_1,\dots),\dots).
\end{align*}

Now let us introduce the addition and multiplication in $W(R)$, denoted by $\oplus$ and $\odot$, respectively.
For any elements $\boldsymbol{a},\boldsymbol{b}\in W(R)$, let
$$\boldsymbol{a}\oplus\boldsymbol{b} = \boldsymbol{c} = (c_0,c_1,\dots),\quad
  \boldsymbol{a}\odot\boldsymbol{b} = \boldsymbol{d} = (d_0,d_1,\dots),$$
where $c_k$ and $d_k$ satisfy $c_k = S_k(a_0, b_0, a_1, b_1, \dots)$ and $d_k = M_k(a_0, b_0, a_1, b_1, \dots)$
 for all $k\in\mathbb{N}_0$.

It can be verified that $W(R)$ is a ring under the operations $\oplus$ and $\odot$, and $W(R)$ is called the ring of Witt vectors or Witt ring over $R$.
The map $\omega$ is a ring homomorphism. If $R$ is a perfect ring of characteristic $p$, then $W(R)$ is a strict $p$-ring with residue ring $R$.

\subsection{Witt rings over finite fields}\label{boxsub}

Throughout the rest of the paper, always assume $R=\F_q$. We will explore the Witt ring $W(\F_q)$ and its relationship with $\Z_q$. Let $T_q$ be the set of Teichm\"{u}ller representatives of $\mathbb{F}_{q}$ in $\mathbb{Z}_q$ and the related Teichm\"{u}ller lifting be $\tau:\mathbb{F}_{q}\rightarrow \Z_q, a\mapsto \tau(a)$. Then $T_q=\{\tau(a) | a\in \F_q\}$, and for each $a\in \F_q$, we have $\tau(a)^q=\tau(a)$ and $\tau(a)\equiv a \pmod{p}$. Conversely, for any $a\in T_q$, let $\widetilde{a}$ be the unique element in $\F_q$ such that $\tau(\widetilde{a})=a$. The ring isomorphism between $W(\F_q)$ and $\Z_q$, also denote by $\tau$, is given by
\begin{equation}\label{zqmap}
  \tau:W(\F_q) \rightarrow \Z_q,\quad  (a_0,a_1^p,a_2^{p^2},\dots)\mapsto \tau(a_0)+\tau(a_1)p+\tau(a_2)p^2+\cdots
\end{equation}
If $q=p$, then $W(\F_p)\cong\Z_p$. Moreover, since $a^p=a$ for $a\in\F_p$, the map \eqref{zqmap} becomes
\begin{equation}\label{zpmap}
  \tau:W(\F_p) \rightarrow \Z_p,\quad  (a_0,a_1,a_2,\dots)\mapsto \tau(a_0)+\tau(a_1)p+\tau(a_2)p^2+\cdots
\end{equation}

Let $m\in \N$. Let $W_m(\F_q)$ denote the finite (or truncated) Witt ring of length $m$ over $\F_q$. The underlying set of $W_m(\F_q)$ is the set $\{(a_0,a_1,\dots,a_{m-1}) \mid a_i\in \F_q\}$. The ring isomorphism between $W_m(\F_q)$ and $\Z_q/p^m \Z_q$, also denote by $\tau$, is given by
\begin{equation}\label{newmap}
  \tau: W_m(\F_q) \rightarrow \Z_q/p^m \Z_q,\quad  (a_0,a_1^p,\dots, a_{m-1}^{p^{m-1}})\mapsto \tau(a_0)+\tau(a_1)p+\cdots+\tau(a_{m-1})p^{m-1}.
\end{equation}
Let $\mathrm{GR}(p^m,h)$ denote the Galois ring of characteristic $p^m$ and cardinality $q^{m}$. It is well known that 
$$\mathrm{GR}(p^m,h)\cong\Z_q/p^m \Z_q\cong W_m(\F_q).$$ 
In particular, $\mathrm{GR}(p,h)\cong \F_q\cong W_1(\F_q)$ and $\mathrm{GR}(p^m,1)\cong \Z/p^m\Z\cong W_m(\F_p).$ 

Let $r\in \N$ and $n\in [0,m-1]$. Write $X_j=(x_{0j},\dots,x_{m-1j})\in W_m(\F_q)$ for $j\in[1,r]$ and define
\begin{equation}\label{sm}
(S_{0}^{(r)},\dots,S_{m-1}^{(r)}):=X_1\oplus\cdots\oplus X_r, \quad (M_{0}^{(r)},\dots,M_{m-1}^{(r)}):=X_1\odot\cdots\odot X_r.
\end{equation}
For a variable $x$, let $\mathrm{wt}(x)$ denote the weighted degree of $x$. The estimates for the upper bounds of degrees of the polynomials $S_n^{(r)}$ and $M_n^{(r)}$ are given below. The proof of the assertion for $S_{n}^{(r)}$ in the special case was given in \cite{caowan}. Here we just prove the assertion for $M_{n}^{(r)}$, and the assertion for $S_{n}^{(r)}$ can be deduced similarly.
\begin{lemma}\label{homcor}
Let $d_j\in \N$ and set $\mathrm{wt}(x_{ij})\leq d_jp^i$ for $i\in [0,n]$ and $j\in [1,r]$. Let $d'=\max\{d_j\mid j\in [1,r]\}$ and $d=\sum_{j=1}^{r}d_j$. Then

\textup{(i)} $S_{n}^{(r)}$ is a polynomial with integer coefficients in $(n+1)r$ variables $x_{ij}(i\in [0,n], j\in [1,r])$, and the weighted degree of $S_{n}^{(r)}$ is not more than $d'p^n$. In particular, if $\mathrm{wt}(x_{ij})=p^i$ for $i\in [0,n]$ and $j\in [1,r]$, then $S_n^{(r)}$ is weighted homogeneous of weighted degree $p^n$.

\textup{(ii)} $M_{n}^{(r)}$ is a polynomial with integer coefficients in $(n+1)r$ variables $x_{ij}(i\in [0,n], j\in [1,r])$, and the weighted degree of $M_{n}^{(r)}$ is not more than $dp^n$. In particular, if $\mathrm{wt}(x_{ij})=p^i$ for $i\in [0,n]$ and $j\in [1,r]$, then $M_{n}^{(r)}$ is weighted homogeneous of weighted degree $rp^n$.
\end{lemma}
\begin{proof}
It is obvious that the polynomial $M_n^{(r)}$ is determined by the first $n+1$ coordinates of $X_j$ with $j\in[1,r]$. From \eqref{sm} and Theorem \ref{theorem:Witt:op:poly} we conclude that the polynomial $M_{n}^{(r)}$ has integer coefficients and that
$
  \omega_n(M_{0}^{(r)},\dots,M_{n}^{(r)})=\prod_{j=1}^{r}\omega_n(X_j),
$
which in expansion by \eqref{wittpoly} is
\begin{equation}\label{epp}
  (M_{0}^{(r)})^{p^n}+p(M_{1}^{(r)})^{p^{n-1}}+\cdots+p^n(M_{n}^{(r)})=\prod_{j=1}^{r}(x_{0j}^{p^n}+p x_{1j}^{p^{n-1}}+\cdots+p^n x_{nj}).
\end{equation}
Thus we have
\begin{equation}\label{ss}
  M_{n}^{(r)}=\frac{1}{p^n}\left(\prod_{j=1}^{r}(x_{0j}^{p^n}+p x_{1j}^{p^{n-1}}+\cdots+p^n x_{nj})-\sum_{i=0}^{n-1}p^i(M_{i}^{(r)})^{p^{n-i}}\right).
\end{equation}
  Since $M_{n}^{(r)}$ has integer coefficients, the factor $\frac{1}{p^n}$ in \eqref{ss} will be cancelled eventually. We make use of induction on $n$ to show that $M_{n}^{(r)}$ is of weighted degree at most $dp^n$ in $(n+1)r$ variables $x_{ij}(i\in [0,n], j\in [1,r])$. The case of $n=0$ in which $M_{0}^{(r)}=x_{01}\cdots x_{0r}$ is trivially verified. We assume that $M_{k}^{(r)}$ is of weighted degree at most $dp^k$ in $(k+1)r$ variables $x_{ij}(i\in [0,k], j\in [1,r])$ for $0\leq k\leq n-1$. From \eqref{ss}, we see that the sum $\sum_{i=0}^{n-1}p^i(M_{i}^{(r)})^{p^{n-i}}$ is of weighted degree at most $dp^{n}$ in $nr$ variables $x_{ij}(i\in [0,n-1], j\in [1,r])$. Note that by \eqref{epp} the product $\prod_{j=1}^{r}\omega_n(X_j)$ in \eqref{ss} is of weighted degree at most $dp^{n}$ in $(n+1)r$ variables $x_{ij}(i\in [0,n], j\in [1,r])$ with the variables $x_{nj}(j\in [1,r])$ not occurring in $\sum_{i=0}^{n-1}p^i(M_{i}^{(r)})^{p^{n-i}}$, which implies that $M_{n}^{(r)}\neq 0$ that will be useful to prove the weighted homogeneous case later. Thus we conclude that $M_{n}^{(r)}$ is of weighted degree at most $dp^n$ in $(n+1)r$ variables $x_{ij}(i\in [0,n], j\in [1,r])$. If $\mathrm{wt}(x_{ij})=p^i$ for $i\in [0,n]$ and $j\in [1,r]$, we can similarly show that $M_{n}^{(r)}$ is weighted homogeneous of weighted degree $rp^n$. 
\end{proof}

Throughout the rest of this subsection, let $X_j=(x_{0j}, x_{1j}^p,\dots,x_{{m-1}j}^{p^{m-1}})\in W_m(\F_q)$ for $j\in[1,r]$. Then $\tau(X_j)=\sum_{i=0}^{m-1}\tau(x_{ij})p^i\in \Z_q/p^m\Z_q$, where $\tau$ denotes the ring isomorphism between $W(R)$ and $\Z_q/p^m\Z_q$ given by \eqref{newmap}. Given a vector $\alpha=(\alpha_1,\dots,\alpha_r)\in \N_0^r$, we set $|\alpha|:=\alpha_1+\cdots+\alpha_r$.

\begin{example}
With the above notations, define
\begin{equation*}
(h_0,\dots,h_{m-1}):=\underbrace{X_1\odot\cdots\odot X_1}_{\alpha_1}\odot\cdots\odot\underbrace{X_r\odot\cdots\odot X_r}_{\alpha_r}.
\end{equation*}
By Lemma \ref{homcor}, $h_n$ with $n\in[0,m-1]$ is a homogeneous polynomial of degree $|\alpha|p^n$ in $(n+1)r$ variables $x_{ij}(i\in [0,n],j\in[1,r])$.
\end{example}

We also need to estimate the degrees of the terms appearing at the most right side in \eqref{newmap} after the multiplications given by \eqref{sm}. The lemma below will be used in the later proof in which the variables $\tau(x_{ij})$ are replaced by the polynomials $g_{ij}$.

\begin{lemma}\label{jjkk}
With the above notations, set $g_{ij}:=\tau(x_{ij})$ and $\mathrm{wt}(g_{ij})\leq p^{h\lfloor\frac{i}{h}\rfloor}$ for $j\in[1,r]$ and $i\in[0,m-1]$. Given a vector $\beta=(\beta_{11},\dots,\beta_{\alpha_{1}1};\dots; \beta_{1r},\dots,\beta_{\alpha_{r}r})\in \N_0^{|\alpha|}$, we write $\beta=(\beta_{tl})$ for short. Then
\begin{equation*}
\deg\left(\prod_{l=1}^{r}\prod_{t=1}^{\alpha_l}g_{_{{\beta_{tl}}l}}\right)\leq |\alpha|p^{h\lfloor\frac{|\beta|}{h}\rfloor}.
\end{equation*}
\end{lemma}
\begin{proof}
Note that $|\beta|=\sum_{l=1}^{n}\sum_{t=1}^{\alpha_l}\beta_{tl}$. The estimate follows immediately from the inequality
\begin{equation*}
\deg\left(\prod_{l=1}^{r}\prod_{t=1}^{\alpha_l}g_{_{{\beta_{tl}}l}}\right)=\sum_{l=1}^{r}\sum_{t=1}^{\alpha_l}\deg(g_{_{{\beta_{tl}}l}})\leq \sum_{l=1}^{r}\sum_{t=1}^{\alpha_l}p^{h\lfloor\frac{\beta_{tl}}{h}\rfloor}
\leq |\alpha|p^{h\lfloor\frac{|\beta|}{h}\rfloor}.
\end{equation*}
\end{proof}

Finally, we present two lemmas that play the crucial role in the proofs of our main results. Let $\widetilde{s}_n^{(r)}$ and $\widetilde{m}_n^{(r)}$ be two functions such that
\begin{equation}\label{smallsm}
\sum_{j=1}^{r}\left(\sum_{i=0}^{m-1}\tau(x_{ij})p^i\right)=\sum_{n=0}^{m-1}\tau(\widetilde{s}_n^{(r)})p^n, \mbox{     and     } \prod_{j=1}^{r}\left(\sum_{i=0}^{m-1}\tau(x_{ij})p^i\right)=\prod_{n=0}^{m-1}\tau(\widetilde{m}_n^{(r)})p^n.
\end{equation}
Since $\widetilde{s}_n^{(r)}$ and $\widetilde{m}_n^{(r)}$ have fraction degrees (cf. \cite{caowan}), we set $s_n^{(r)}:=(\widetilde{s}_n^{(r)})^{p^n}$ and $m_n^{(r)}:=(\widetilde{m}_n^{(r)})^{p^n}$.

\begin{lemma}[\cite{caowan}]\label{smm}
The polynomials $s_n^{(r)}$ and $m_n^{(r)}$ have integer coefficients. Moreover, $s_n^{(r)}$ is homogeneous of degree $p^n$ and $m_n^{(r)}$ is homogeneous of degree $rp^n$ in $(n+1)r$ variables:
    \begin{align*}
s_n^{(r)}& =S_n^{(r)}(x_{01},\dots,x_{0r};x_{11}^{p},\dots,x_{1r}^{p};\dots;x_{n1}^{p^n},\dots,x_{nr}^{p^n}), \mbox{  and  } \\
m_n^{(r)}& =M_n^{(r)}(x_{01},\dots,x_{0r};x_{11}^{p},\dots,x_{1r}^{p};\dots;x_{n1}^{p^n},\dots,x_{nr}^{p^n}).
  \end{align*}
\end{lemma}

\begin{lemma}[\cite{caowan}]\label{keylem}
 Let $r\in \N$ and $\sum_{i=0}^{m-1}\tau(x_{ij})p^i\in \Z_q/p^m\Z_q$ with $x_{ij}\in \F_q, j\in [1,r]$. Suppose $\sum_{j=1}^{r}\left(\sum_{i=0}^{m-1}\tau(x_{ij})p^i\right)=\sum_{n=0}^{m-1}\tau(\widetilde{s}_n^{(r)})p^n$. For $n\in [0,m-1]$, let $s_n^{(r)}:=(\widetilde{s}_n^{(r)})^{p^n}$. Then the following statements are equivalent:

\textup{(i)}  $\sum_{j=1}^{r}\sum_{i=0}^{m-1}\tau(x_{ij})p^i\equiv 0 \mod {p^m}$.

\textup{(ii)}  $\widetilde{s}_0^{(r)}=\widetilde{s}_1^{(r)}=\dots=\widetilde{s}_{m-1}^{(r)}=0$.

\textup{(iii)} $s_0^{(r)}=s_1^{(r)}=\dots=s_{m-1}^{(r)}=0$.
\end{lemma}

\section{Main results}\label{sect3}

A polynomial $g$ over $\Z_q$ is called a Teichm\"uller polynomial if all of its coefficients are
Teichm\"uller elements in $T_q$, and $g$ is called reduced if its degree in each variable is at most $q-1$. Given a Teichm\"uller polynomial $g(X)=\sum_{j=1}^{t}b_jX^{u_j}\in\Z_q[x_1,\dots, x_n]$ with $b_j\in T_q$ for all $j\in[1,t]$, we define $\widetilde{g}(X):=\sum_{j=1}^{t}\widetilde{b}_jX^{u_j}\in\F_q[x_1,\dots, x_n]$. The following lemma describes the box $\mathcal B_m$ in terms of the image of a reduced Teichm\"uller polynomial system, and the proof for $m=1$ can be found in \cite[Lemma 4.1]{caowan}. Let $X_m$ denote the collection of $nm$ variables $x_{ij}$ with $i\in[0,m-1],j\in[1,n]$.
\begin{lemma}\label{boxlem}
Let $\mathcal B_m$ be a subset of $\Z_q^n$ with $q^{nm}$ elements such that $\mathcal B_m$ modulo $p^m$ is equal to $(\Z_q/p^m \Z_q)^n$. Then there exists a unique system of reduced Teichm\"uller polynomials $g_{ij}(X_m)$ $(i\geq m,n\geq j\geq 1)$ in $nm$ variables with total degree bounded by $nm(q-1)$ such that for any $Y\in \mathcal B_m$, we have
\begin{equation}\label{zz1intro}
     Y=X_m + (g_{m1}(X_m),\dots,g_{mn}(X_m))p^m+ (g_{m+11}(X_m),\dots,g_{m+1n}(X_m))p^{m+1}+\cdots,
\end{equation}
where $X_m=\sum_{i=0}^{m-1}(x_{i1},\dots,x_{in})p^i\in \mathcal T_m$ is the Teichm\"uller lifting of the modulo $p^m$ reduction of $Y$. In particular, if $\mathcal B_m$ is in split form, then each $g_{ij}$ has total degree at most $m(q-1)$.
\end{lemma}
\begin{proof} 
For a given $Y=(y_1, \dots, y_n)\in \mathcal B_m$, we can write uniquely
\begin{equation*}
  (y_1, \dots, y_n)=\sum\nolimits_{i=0}^{m-1}(x_{i1},\dots,x_{in})p^i+(x_{m1}, \dots, x_{mn})p^m,
\end{equation*}
where $X_m=\sum_{i=0}^{m-1}(x_{i1},\dots,x_{in})p^i \in \mathcal T_m$ and $(x_{m1}, \dots, x_{mn}) \in \Z_q^n$.
The vector $X_m$ in $\mathcal T_m$ and the vector $Y=(y_1, \dots, y_n)$ in $\mathcal B_m$ determine each other.
In fact, $X_m$ is just the Teichm\"uller lifting of the reduction $Y \mod p^m$, and $Y$ is the unique element in $\mathcal B_m$ with the
same mod $p^m$ reduction as $X_m$. In particular, $(x_{m1}, \dots, x_{mn})$ is also uniquely determined by
 $X_m=\sum_{i=0}^{m-1}(x_{i1},\dots,x_{in})p^i$.

Letting $Y$ run over $\mathcal B_m$, then $X_m$ runs over $\mathcal T_m$ as $\mathcal B_m \mod p^m=(\Z_q/p^m \Z_q)^n$ by assumption.
For each $1\leq j \leq n$, the quantity $x_{mj}$ is a function of $X_m$. This establishes a map from $\mathcal T_m$ to $\Z_q$, and we consider the corresponding map from $\F_q^{nm}$ to $\F_q$, namely
 \begin{equation*}
  \widetilde{g}_{mj}:\quad \F_q^{nm} \rightarrow \F_q,\quad  X_m\mapsto \widetilde{g}_{mj}(X_m).
\end{equation*}
Recall the fact that any map from $\F_q^{nm}$ to $\F_q$ can be expressed uniquely by a reduced polynomial in $nm$ variables with coefficients in $\F_q$. In particular, our map $\widetilde{g}_{mj}$ is a reduced polynomial in $\mathbb{F}_q[x_{01},\ldots, x_{0n},\ldots,x_{m-11},\ldots, x_{m-1n}]$.
Let $g_{mj}$ be the Teichm\"uller lifting of $\widetilde{g}_{mj}$ in $\Z_q[x_{01},\ldots, x_{0n},\ldots,x_{m-11},\ldots, x_{m-1n}]$. Then, we have proved
$$Y = X_m + (g_{m1}(X_m), \cdots, g_{mn}(X_m))p^m + (x_{m+11}, \cdots, x_{m+1n})p^{m+1},$$
where $(x_{m+11}, \cdots, x_{m+1n}) \in \Z_q^n$ is uniquely determined by $X_m$. Continuing this procedure, we find
uniquely determined reduced Teichm\"uller polynomials $g_{ij} \in \Z_q[x_{01},\ldots, x_{0n},\ldots,x_{m-11},\ldots, x_{m-1n}]$ ($i\geq m$, $1\leq j\leq n$) such that \eqref{zz1intro} holds. Since each variable in $g_{ij}$ has degree at most $q-1$, the total degree of $g_{ij}$ is at most $nm(q-1)$. In particular, if $\mathcal B_m$ is in split form, then $g_{ij} \in \Z_q[x_{0j},\ldots, x_{m-1j}]$ and hence $\deg (g_{ij})\leq m(q-1)$. 
The lemma is proved.
\end{proof}

\begin{notation}
 For convenience, we set $g_{ij}=x_{ij}$ for $i\in [0,m-1], j\in [1,n]$, and simply write $\mathcal B_m=\mathcal T_m[g_{ij}:i\in \N_0,j\in [1,n]]$.
\end{notation}

\begin{defn}
Let $m'\in \N$. If $\deg( g_{ij})\leq p^{h\lfloor\frac{i}{h}\rfloor}$ for all $i\in[0,m'-1],j\in[1,n]$, we call that the box $\mathcal B_m$ is $m'$-algebraically close to the Teichm\"uller box $\mathcal T_m$, and denote by $\mathcal B_m \overset{m'}{\sim} \mathcal T_m$.  
\end{defn}

\begin{example}
Below are three common examples for $\mathcal B_m\overset{m'}{\sim} \mathcal T_m$:
\begin{enumerate} 
  \item $\mathcal B_m\overset{m'}{\sim} \mathcal T_m$ for all $1\leq m'\leq m$ because $\deg(g_{ij})=1$ for all $i\in [0,m'-1],j\in [1,n]$.
  \item If $\mathcal B_1$ is in split form, then $\mathcal B_1\overset{m'}{\sim} \mathcal T_1$ for any $m'\in \N$ because $\deg(g_{ij})\leq q-1$ for all $i\geq 1,j\in [1,n]$ by Lemma \ref{boxlem}.
  \item If $\mathcal B_m=\mathcal T_m$, then $\mathcal B_m\overset{m'}{\sim} \mathcal T_m$ for any $m'\in \N$  because $\deg(g_{ij})=0$ for all $i\geq m,j\in [1,n]$. Here we take the degree of the zero polynomial to be zero.
\end{enumerate}
\end{example}

Fix a nonzero vector $\alpha=(\alpha_1,\dots,\alpha_n)\in \N_0^n$. For an element $\beta\in \N_0^{|\alpha|}$ of the form
$
  \beta=(\beta_{11},\dots,\beta_{\alpha_{1}1};\beta_{12},\dots,\beta_{\alpha_{2}2};\dots; \beta_{1n},\dots,\beta_{\alpha_{n}n}),
$
we write $\beta=(\beta_{tl})$ for short. Let $X_m$ denote the collection of $nm$ variables $x_{ij}$ with $i\in[0,m-1],j\in[1,n]$. By \eqref{zz1intro}, we have $Y=(y_1,\dots,y_n)=\sum_{i=0}^{\infty}(g_{i1}(X_m),\dots,g_{in}(X_m))p^i$ with $y_j=\sum_{i=0}^{\infty}g_{ij}(X_m)p^i$ and $g_{ij}$ being the Teichm\"uller polynomials in $nm$ variables. Then
\begin{equation}\label{77cc}
\prod_{l=1}^{n}\left(\sum_{k=0}^{\infty}g_{_{kl}}p^k\right)^{\alpha_l}=
\sum\prod_{l=1}^{n}\prod_{t=1}^{\alpha_l}\left(g_{_{{\beta_{tl}}l}}p^{\beta_{tl}}\right)=
\sum\left(\prod_{l=1}^{n}\prod_{t=1}^{\alpha_l}g_{_{{\beta_{tl}}l}}\right)p^{|\beta|},
\end{equation}
where both the sums in \eqref{77cc} run over all the vectors $\beta=(\beta_{tl})\in \N_0^{|\alpha|}$. Now let $f(X)=\sum_{j=1}^{r}a_jX^{u_j}\in\Z_q[x_1,\dots, x_n]$. Suppose $a_j=\sum_{i=1}^{\infty}a_{ij}p^i$ with $a_{ij}\in T_q$. Substituting $Y$ for $X$, by \eqref{77cc} we get that the Teichmüller expansion of $f(Y)$ is
\begin{equation}\label{thmus}
f(Y)=\sum_{j=1}^{r}a_{ij}\sum_{\beta=(\beta_{tl})\in \N_0^{|u_j|}}\left(\prod_{l=1}^{n}\prod_{t=1}^{\alpha_l}g_{_{{\beta_{tl}}l}}\right)p^{i+|\beta|}.
\end{equation}
Note that each $g_{ij}$ is a polynomial in $nm$ variables, so be $f(Y)$.

\subsection{For a single polynomial}
We first consider the single polynomial case, which is relatively easy to deal with and can be extended to the system of polynomials without too much difficulties. 
\begin{thm}\label{thmforsing}
Let $p$ be a prime number and $q = p^h$ with $h\in\mathbb{N}$. Let $m,m'\in \N$. Let $\mathcal B_m=\mathcal T_m[g_{ij}:i\in \N_0,j\in [1,n]]$. Let $f\in \mathbb{Z}_q[x_1,\ldots,x_n]$ be a nonconstant polynomial. Let $V=\{X\in \mathcal B_m: f(X)\equiv 0\mod {p^{m'}}\}$. Assume $\mathcal B_m \overset{m'}{\sim} \mathcal T_m$. In general, we have
\begin{equation}\label{sing4}
\mathrm{ord}_q(|V|)\geq\left\lceil\frac{nm-\frac{p^{m'}-1}{p-1}\deg f}{p^{m'-1}\deg f}\right\rceil^*.
\end{equation}
In particular, further assume $m\geq m'$, then
\begin{numcases}{\mathrm{ord}_q(|V|)\geq}
 \left\lceil n/\deg f-1\right\rceil^*+n(m-m') & if   $m'=1$,\nonumber\\
   \left\lfloor(nm'-1)/2\right\rfloor+n(m-m') &  if $m'>1 \mbox{ and }\deg f> 1$,\label{sing2}\\
  nm-m' & \mbox{otherwise.}\nonumber
\end{numcases}
\end{thm}
\begin{proof}
Let $f(X)=\sum_{j=1}^{r}a_jX^{u_j}\in\Z_q[x_1,\dots, x_n]$ and assume that $f$ has the Teichm\"uller expansion as in \eqref{thmus}. For each $i\in \N_0, j\in[1,r]$ and $\beta=(\beta_{tl})\in \N_0^{|u_j|}$ with $i+|\beta|<m'$, the term $a_{ij}\prod_{l=1}^{n}\prod_{t=1}^{\alpha_l}g_{_{{\beta_{tl}}l}}p^{i+|\beta|}$ corresponds to the $(i+|\beta|)$-th component in $W_{m'}(\F_q)$.
Given an element $Y\in \mathcal B_m$, by Lemma \ref{boxlem} we have $Y=\sum_{i=0}^{\infty}(g_{i1}(X_m),\dots,g_{in}(X_m))p^i$ with $X_m\in \mathcal T_m$ and $g_{ij}$ being the reduced Teichm\"uller polynomials in $nm$ variables. Then by Lemma \ref{keylem} we have $f(Y)\equiv 0\mod {p^{m'}}$ if and only if for all $k\in [0,m'-1]$,
$$h_k(\widetilde{X}_m):=s_k^{(r)}\bigg(\Big(\widetilde{a}_{ij}\prod_{l=1}^{n}\prod_{t=1}^{\alpha_l}
\widetilde{g}_{_{{\beta_{tl}}l}}(\widetilde{X}_m)\Big)^{p^{i+|\beta|}}\; \Big | \; j\in [1,r], \beta=(\beta_{tl})\in \N_0^{|u_j|}, i+|\beta|\leq k \bigg)=0,$$
where $\widetilde{a}_{ij}\in \F_q, \widetilde{X}_m\in \F_q^{mn}$ and $\widetilde{g}_{_{{\beta_{tl}}l}}$ are the polynomials over $\F_q$ in $nm$ variables. Define
  \begin{equation*}
  \widetilde{V}:=\left\{X_m\in \F_q^{nm} \mid  \widetilde{h}_k(X_m)= 0 \mbox{ for all } k\in [0,m'-1]\right\}.
  \end{equation*}
Then $|V|=|\widetilde{V}|$. By assumption $\mathcal B_m \overset{m'}{\sim} \mathcal T_m$, i.e., $\deg(g_{ij})\leq p^{h\lfloor\frac{i}{h}\rfloor}$ for all $i\in[0,m'-1],j\in[1,n]$. Note $q=p^h$ and hence $\alpha^{p^c}=\alpha^{p^{c-h\lfloor\frac{c}{h}\rfloor}}$ in $\F_q$ for any $\alpha\in \F_q$ and $c\in \N_0$. Taking arbitrarily a term in $\widetilde{h}_k$, say $\Big(\widetilde{a}_{ij}\prod_{l=1}^{n}\prod_{t=1}^{\alpha_l}\widetilde{g}_{_{{\beta_{tl}}l}}(\widetilde{X}_m)\Big)^{p^{i+|\beta|}}$, we have
\begin{equation*}
\Big(\widetilde{a}_{ij}\prod_{l=1}^{n}\prod_{t=1}^{\alpha_l}
\widetilde{g}_{_{{\beta_{tl}}l}}\Big)^{p^{i+|\beta|}}=\Big(\widetilde{a}_{ij}\prod_{l=1}^{n}\prod_{t=1}^{\alpha_l}
\widetilde{g}_{_{{\beta_{tl}}l}}\Big)^{p^{i+|\beta|-h\lfloor \frac{|\beta|}{h}\rfloor}}.
\end{equation*}
It follows from Lemma \ref{jjkk} that
\begin{equation}\label{sunak}
\deg\Big(\prod_{l=1}^{n}\prod_{t=1}^{\alpha_l}
g_{_{{\beta_{tl}}l}}\Big)p^{(i+|\beta|-h\lfloor \frac{|\beta|}{h}\rfloor)}\leq |u_j|p^{i+|\beta|}\leq p^{i+|\beta|} \deg f.
\end{equation}
Thus by Lemma \ref{homcor}, $\deg\widetilde{h}_k\leq p^k\deg f$ for $k\in [0,m'-1]$. Consequently,
$$\sum_{i=0}^{m'-1}\deg\widetilde{h}_k\leq\sum_{i=0}^{m'-1}p^k\deg f= \frac{p^{m'}-1}{p-1}\deg f.$$
Note $|V|=|\widetilde{V}|$. Applying the Ax-Katz theorem to $\widetilde{V}$ yields
 \begin{equation*}
  \mathrm{ord}_q(|V|)=\mathrm{ord}_q(|\widetilde{V}|)\geq \left\lceil\frac{nm-\sum_{k=0}^{m'-1}\deg\widetilde{h}_k}{\max\nolimits_{k\in [0,m'-1]}\deg\widetilde{h}_k}\right\rceil^*\geq \left\lceil\frac{nm-\frac{p^{m'}-1}{p-1}\deg f}{p^{m'-1}\deg f}\right\rceil^*.
  \end{equation*}
This completes the proof of \eqref{sing4}.

Now assume $m>m'\geq 1$. Define $\mathcal B_{m'}$ to be a subset of $\Z_q^n$ with $q^{nm'}$ elements such that $\mathcal B_{m'}$ modulo $p^{m'}$ is equal to $(\Z_q/p^{m'} \Z_q)^n$ and $V'=\{X\in \mathcal B_{m'}: f(X)\equiv 0\mod {p^{m'}}\}$. By Theorems \ref{axkatzthm} and \ref{mrkthm}, we have 
\begin{numcases}{\mathrm{ord}_q(|V'|)\geq}
 \left\lceil n/\deg f-1\right\rceil^* & if   $m'=1$,\nonumber\\
   \left\lfloor(nm'-1)/2\right\rfloor &  if $m'>1,n>1 \mbox{ and }\deg f> 1$,\label{sing2a}\\
  (n-1)m' & \mbox{otherwise.}\nonumber
\end{numcases}
Suppose $V'\neq \emptyset$. Given an arbitrary  $X=\sum_{i=0}^{m'-1}A_ip^i+\sum_{j=m'}^{m-1}B_jp^j+\sum_{k=m}^{\infty}C_kp^k\in V'$ with $A_i, B_j, C_k\in T_q^n$, if we replace $B_j(j\in[m',m-1])$ by any element $B_j'$ in $T_q^n$, then it is trivial to see that $X'=\sum_{i=0}^{m'-1}A_ip^i+\sum_{j=m'}^{m-1}B_j'p^j+\sum_{k=m}^{\infty}C_kp^k\in V$. Thus $\mathrm{ord}_q(|V|)=\mathrm{ord}_q(|V'|)+n(m-m')$, and \eqref{sing2} follows immediately by \eqref{sing2a}. Note that \eqref{sing2} also holds for the case of $V'=\emptyset$.
\end{proof}

Theorem \ref{thmforsing} for the general estimate \eqref{sing4} can be improved provided that we give a tighter upper bound in \eqref{sunak}.
\begin{thm}\label{thmforsing2}
Let $p$ be a prime number and $q = p^h$ with $h\in\mathbb{N}$. Let $m,m'\in \N$. Let $\mathcal B_m=\mathcal T_m[g_{ij}:i\in \N_0,j\in [1,n]]$. Let $f\in \mathbb{Z}_q[x_1,\ldots,x_n]$ be a nonconstant polynomial. Let $V=\{X\in \mathcal B_m: f(X)\equiv 0\mod {p^{m'}}\}$. Write the Teichm\"uller expansion $f=\sum_{i=0}^{\infty} p^i f_i$ with $f_i=\sum_{j=1}^{r}a_{ij}X^{u_j}$. Let $d\in \N$ with $d\leq \deg f$. If for each term $a_{ij}X^{u_j}$, we have
\begin{equation}\label{degreesum}
\deg\left(a_{ij}\prod_{l=1}^{n}\prod_{t=1}^{\alpha_l}g_{_{{\beta_{tl}}l}}(X)\right)\leq dp^{h\lfloor\frac{i+|\beta|}{h}\rfloor}
\end{equation}
for all $i\in [0,m'-1]$, $j\in [1,r]$,  $\beta=(\beta_{tl})\in \N_0^{|u_j|}$ with the sum $i+|\beta|\leq m'-1$, then
\begin{equation*}
\mathrm{ord}_q(|V|)\geq \left\lceil\frac{nm-\frac{p^{m'}-1}{p-1}d}{p^{m'-1}d}\right\rceil^*.
\end{equation*}
\end{thm}
\begin{proof}
Note that in $\F_q$ with $q=p^h$, we have
\begin{equation*}
\Big(\widetilde{a}_{ij}\prod_{l=1}^{n}\prod_{t=1}^{\alpha_l}
\widetilde{g}_{_{{\beta_{tl}}l}}\Big)^{p^{i+|\beta|}}=\Big(\widetilde{a}_{ij}\prod_{l=1}^{n}\prod_{t=1}^{\alpha_l}
\widetilde{g}_{_{{\beta_{tl}}l}}\Big)^{p^{i+|\beta|-h\lfloor \frac{i+|\beta|}{h}\rfloor}}.
\end{equation*}
By the assumption \eqref{degreesum}, one sees that \eqref{sunak} becomes
\begin{equation*}
\deg\Big(\prod_{l=1}^{n}\prod_{t=1}^{\alpha_l}
g_{_{{\beta_{tl}}l}}\Big)p^{(i+|\beta|-h\lfloor \frac{i+|\beta|}{h}\rfloor)}\leq dp^{i+|\beta|}.
\end{equation*}
The other is similar to the proof of Theorem \ref{thmforsing}.
\end{proof}

\subsection{For a polynomial system}
Now we consider the polynomial system case. As mentioned above, it can be extended from Theorem \ref{thmforsing} without much more difficulties except for more cumbersome notation. So we only emphasize the difference between them.
\begin{thm}\label{thmforsystem}
Let $p$ be a prime number and $q = p^h$ with $h\in\mathbb{N}$. Let $m,m_1,\dots,m_s\in \N$ with $m_s\geq \dots \geq m_1$. Let $\mathcal B_m=\mathcal T_m[g_{ij}:i\in \N_0,j\in [1,n]]$. Let $f_1,\dots,f_s\in \mathbb{Z}_q[x_1,\ldots,x_n]$ be a system of nonconstant polynomials. Let $V=\{X\in \mathcal B_m: f_k(X)\equiv 0\mod {p^{m_k}}\mbox{ for all } k\in [1,s]\}$. Assume $\mathcal B_m \overset{m_s}{\sim} \mathcal T_m$. In general, we have
\begin{equation}\label{gg07}
  \mathrm{ord}_q(|V|)\geq\left\lceil\frac{nm-\sum_{k=1}^s\frac{p^{m_k}-1}{p-1}\deg f_k}{\max\nolimits_{k \in [1,s]}\{p^{m_k-1}\deg f_k\}}\right\rceil^*.
\end{equation}
In particular, further assume that $m\geq m_s=\dots=m_1$, then we have
\begin{numcases}{\mathrm{ord}_q(|V|)\geq}
 \left\lceil\frac{n-\sum_{k=1}^s\deg f_k}{\max\nolimits_{k \in [1,s]}\deg f_k}\right\rceil^*+n(m-m_1) & if   $m_1=1$,\nonumber \\
  \left\lfloor\frac{(n-s+1)m_1-1}{2}\right\rfloor+n(m-m_1) &  if $m_1>1, n>s$, \text{ and any } $\deg f_k>1$,\label{systempoly2} \\
     \left\lceil(n-s)m_1\right\rceil^*+n(m-m_1) & \mbox{otherwise.} \nonumber
\end{numcases}
\end{thm}
\begin{proof}
We first consider the general estimate. From the proof of Theorem \ref{thmforsing}, we see that for each modulus $p^{m_k}$, the polynomial $f_k$ contributes $m_k$ polynomials $h_{tk}$ over $\F_q$ ($t\in [0,{m_k}-1]$)  whose degree is bounded by $p^{t}\deg f_k$ and thus
  $\sum_{t=0}^{{m_k}-1}\deg h_{tk}\leq \frac{p^{m_k}-1}{p-1}\deg f_k$. Now given $s$ polynomials $f_k$ and $s$ moduli $p^{m_k}$, $k\in[1,s]$, we get $\sum_{k=1}^{s}m_k$ polynomials over $\F_q$ with the sum of degrees $\leq \sum_{k=1}^{s}\frac{p^{m_k}-1}{p-1}\deg f_k$ and maximal degree bounded by $\max\nolimits_{k \in [1,s]}\{p^{m_k-1}\deg f_k\}$. Applying the Ax-Katz theorem, we obtain the desired estimate \eqref{gg07}.
  
  Now assume $m\geq m_s=\dots=m_1$. Define $\mathcal B_{m_1}$ to be a subset of $\Z_q^n$ with $q^{nm_1}$ elements such that $\mathcal B_{m_1}$ modulo $p^{m_1}$ is equal to $(\Z_q/p^{m_1} \Z_q)^n$ and $V'=\{X\in \mathcal B_{m_1}: f(X)\equiv 0\mod {p^{m_1}}\}$. By Theorems \ref{axkatzthm} and \ref{mrkthm}, we have 
\begin{numcases}{\mathrm{ord}_q(|V|)\geq}
 \left\lceil\frac{n-\sum_{k=1}^s\deg f_k}{\max\nolimits_{k \in [1,s]}\deg f_k}\right\rceil^*  & if   $m_1=1$,\nonumber \\
  \left\lfloor\frac{(n-s+1)m_1-1}{2}\right\rfloor  &  if $m_1>1, n>s$, \text{ and any } $\deg f_k>1$,\label{systempolyww2} \\
     \left\lceil(n-s)m_1\right\rceil^*  & \mbox{otherwise.} \nonumber
\end{numcases}
Suppose $V'\neq \emptyset$. Given an arbitrary  $X=\sum_{i=0}^{m_1-1}A_ip^i+\sum_{j=m_1}^{m-1}B_jp^j+\sum_{k=m}^{\infty}C_kp^k\in V'$ with $A_i, B_j, C_k\in T_q^n$, if we replace $B_j(j\in[m_1,m-1])$ by any element $B_j'$ in $T_q^n$, then it is trivial to see that $X'=\sum_{i=0}^{m_1-1}A_ip^i+\sum_{j=m_1}^{m-1}B_j'p^j+\sum_{k=m}^{\infty}C_kp^k\in V$. Thus $\mathrm{ord}_q(|V|)=\mathrm{ord}_q(|V'|)+n(m-m_1)$, and \eqref{systempoly2} follows by \eqref{systempolyww2}. Note that \eqref{systempolyww2} also holds for $V'=\emptyset$.
\end{proof}

The following theorem improves Theorem \ref{thmforsystem} for the general estimate \eqref{gg07}. Its proof is omitted since it is similar to that of Theorem \ref{thmforsing2}.
\begin{thm}
Let $p$ be a prime number and $q = p^h$ with $h\in\mathbb{N}$. Let $m,m_1,\dots,m_s\in \N$. Let $\mathcal B_m=\mathcal T_m[g_{ij}:i\in \N_0,j\in [1,n]]$. Let $f_1,\dots,f_s\in \mathbb{Z}_q[x_1,\ldots,x_n]$ be a system of nonconstant polynomials. Let $V=\{X\in \mathcal B_m: f_k(X)\equiv 0\mod {p^{m_k}}\mbox{ for all } k\in [1,s]\}$. For each $k\in[1,s]$, write the $p$-adic Teichm\"uller expansion
$f_k=\sum_{i=0}^{\infty} p^if_{k,i}(X) $
with $f_{k,i}(X)=\sum_{j=1}^{r_k}a_{ij}^{(k)}X^{u_j^{(k)}}$. Let $d_1,\dots,d_s\in \N$. If for each term $a_{ij}^{(k)}X^{u_j^{(k)}}$, we have
\begin{equation*}
\deg\left(a_{ij}^{(k)}\prod_{l=1}^{n}\prod_{t=1}^{\alpha_l}g_{_{{\beta_{tl}}l}}(X)\right)\leq d_kp^{h\lfloor\frac{i+|\beta|}{h}\rfloor}
\end{equation*}
for all $i\in [0,m_k-1]$, $j\in [1,r_k]$,  $\beta=(\beta_{tj})\in \N_0^{|u_j^{(k)}|}$ with the sum $i+|\beta|\leq m_k-1$, then
\begin{equation*}
\mathrm{ord}_q(|V|)\geq \left\lceil\frac{nm-\sum_{k=1}^s\frac{p^{m_k}-1}{p-1}d_k}{\max\nolimits_{k \in [1,s]}\{p^{m_k-1}d_k\}}\right\rceil^*.
\end{equation*}
\end{thm}

\section{Examples and future research} \label{sect4}
We provide several examples for the single polynomial case, from which we see that our estimate given by \eqref{sing4} can be achieved and that there are still some research worth doing.

\begin{example}
  Let $p=2$ and $h=1$. Let $\mathcal B_2=\mathcal T_2[g_{ij}:i\in \N_0,j\in [1,4]]$ with  $g_{ij} = x_{ij}$ for $i\in[0,1], j\in[1,4]$ and
    $g_{21} = x_{01}x_{11}x_{02}x_{14}, g_{22}=g_{23}=g_{24}=0$ and $g_{ij}=0$ for all $i>2$. Then $\mathcal B_2 \overset{3}{\sim} \mathcal T_2$. Let $f(x_1,x_2,x_3,x_4)=x_1 + 3x_2 + 5x_3+ 6x_4$ be a polynomial over $\mathbb{Z}_2$. Let $V=\{(x_1,x_2,x_3,x_4)\in \mathcal B_2: f(x_1,x_2,x_3,x_4)\equiv 0\mod {2^{3}}\}$. By calculation, we get $|V|=30$ and hence $\ord_2(|V|)= 1$, which achieves the estimate given by \eqref{sing4}. 
\end{example}

As mentioned in Section \ref{sect1}, it is known that the estimate given by \eqref{sing4} is best possible for $m=1$. The example above shows that it is also best possible for $m\geq 2$. However, this estimate may be improved for the other cases, as shown in the example below.
\begin{example}
  Let $p=2$ and $h=1$. Let $\mathcal B_2=\mathcal T_2[g_{ij}:i\in \N_0,j\in [1,4]]$ with  $g_{ij} = x_{ij}$ for $i\in[0,1], j\in[1,4]$ and
    $g_{21} = x_{01}x_{11}x_{02}x_{14}, g_{22} = x_{01}x_{12}x_{03}x_{04}, g_{23}=g_{24}=0$ and $g_{ij}=0$ for all $i>2$. Then $\mathcal B_2 \overset{3}{\sim} \mathcal T_2$. Let $f(x_1,x_2,x_3,x_4)=x_1 + 2x_2 + 2x_3+ 4x_4$ be a polynomial over $\mathbb{Z}_2$. Let $V=\{(x_1,x_2,x_3,x_4)\in \mathcal B_2: f(x_1,x_2,x_3,x_4)\equiv 0\mod {2^{3}}\}$. By calculation, we get $|V|=32$ and hence $\ord_2(|V|)= 5>1$, which is greater than the estimate given by \eqref{sing4}.
\end{example}

Finally, we present an example to show that though the condition that $\mathcal B_m \overset{m'}{\sim} \mathcal T_m$ is important in the proof of Theorems \ref{thmforsing} and \ref{thmforsystem}, it is not necessary.
\begin{example}
  Let $p=2$ and $h=1$. Let $\mathcal B_2=\mathcal T_2[g_{ij}:i\in \N_0,j\in [1,4]]$ with  $g_{ij} = x_{ij}$ for $i\in[0,1], j\in[1,4]$ and
    $g_{21} = x_{01}x_{11}x_{02}x_{12}x_{03}, g_{22} = x_{01}x_{12}x_{03}x_{04}, g_{23}=g_{24}=0$ and $g_{ij}=0$ for all $i>2$. Then the condition $\mathcal B_2 \overset{3}{\sim} \mathcal T_2$ does not hold. Let $f(x_1,x_2,x_3,x_4)=x_1+3x_2+4x_3+7x_4$ be a polynomial over $\mathbb{Z}_2$. Let $V=\{(x_1,x_2,x_3,x_4)\in \mathcal B_2: f(x_1,x_2,x_3,x_4)\equiv 0\mod {2^{3}}\}$. By calculation, we get $|V|=30$ and hence $\ord_2(|V|)\geq 1$.
\end{example}

\end{document}